\documentclass[a4paper,11pt]{article}
\usepackage[utf8]{inputenc}

\usepackage{amsthm,amsmath,amsfonts,amssymb}
\usepackage[english]{babel}
\usepackage{graphicx}
\usepackage{verbatim}

\DeclareMathOperator{\diag}{diag}
\DeclareMathOperator{\dist}{dist}

\DeclareMathOperator{\spec}{sp}

\newcommand{\mytilde}{\raise.17ex\hbox{$\scriptstyle\mathtt{\sim}$}}

\def\dist{\mathop{\rm dist }\nolimits}
\def\diag{\mathop{\rm diag }\nolimits}

\def\z{\mbox{\boldmath $z$}}

\def\vec0{\mbox{\boldmath $0$}}

\def\A{\mbox{\boldmath $A$}}

\def\A{\mbox{\boldmath $A$}}

\def\D{\mbox{\boldmath $D$}}

\def\I{\mbox{\boldmath $I$}}
\def\J{\mbox{\boldmath $J$}}

\def\M{\mbox{\boldmath $M$}}

\def\V{\mbox{\boldmath $V$}}

\def\I{\mbox{\boldmath $I$}}
\def\J{\mbox{\boldmath $J$}}
\def\M{\mbox{\boldmath $M$}}

\theoremstyle{plain}   

\newtheorem{theorem}{Theorem}[section]
\newtheorem{proposition}[theorem]{Proposition}

\newtheorem{problem}[theorem]{Problem}

\begin{document}

\title{On bipartite mixed graphs
}

\author{C. Dalf\'o$^a$, M. A. Fiol$^b$, N. L\'opez$^c$
\\ \\
{\small $^{a,b}$Dep. de Matem\`atiques, Universitat Polit\`ecnica de Catalunya} \\
{\small $^b$Barcelona Graduate School of Mathematics} \\
{\small Barcelona, Catalonia} \\
{\small {\tt\{cristina.dalfo,miguel.angel.fiol\}@upc.edu}} \\
{\small $^c$Dep. de Matem\`atica, Universitat de Lleida}\\
 {\small Lleida, Spain}\\
 {\small {\tt nlopez@matematica.udl.es}}
}
\date{}

\maketitle
\begin{abstract}
Mixed graphs can be seen as digraphs that have both arcs and edges (or digons,
that is, two opposite arcs). In this paper, we consider the case where such
graphs are bipartite. As main results, we show that in this context the
Moore-like bound is attained in the case of diameter $k=3$, and that bipartite
mixed graphs of diameter $k\ge 4$ do not exist.
\end{abstract}



\section{Introduction}
It is well known that the choice of the interconnection network for a multicomputer
or any complex system  is one of the crucial problems the designer has to face.
Indeed, the network topology affects largely the performance of the system and it
has an important contribution to its overall cost. As such topologies are modeled
by either graphs, digraphs, or mixed graphs, this has lead to the following
optimization problems:
\begin{enumerate}
\item
Find graphs, digraphs or mixed graphs, of given diameter and maximum out-degree that have a large number of vertices.
\item
Find graphs, digraphs or mixed graphs, of given number of vertices and maximum out-degree that have small diameter.
\end{enumerate}

For a more detailed description of the problem, its possible applications, the usual notation, and the theoretical background see the comprehensive survey of Miller and \v{S}ir\'a\v{n} \cite{ms13}. For more specific results concerning mixed graphs, which are the topic of this paper, see, for example, Nguyen and Miller \cite{nm08}, and Nguyen, Miller, and Gimbert \cite{nmg07}.

A mixed graph can be seen as a type of digraph containing some edges (two opposite
arcs). Thus, a {\em mixed graph} $G$ with vertex set $V$ may contain (undirected) {\em
edges} as well as directed edges (also known as {\em arcs}). From this point of
view, a {\em graph} (respectively, {\em directed graph} or {\em digraph}) has all its edges
undirected (respectively, directed). In fact, we can identify the mixed graph $G$ with its
associated digraph $G^*$ obtained by replacing all the edges by digons (two
opposite arcs or a directed $2$-cycle).
The {\em undirected degree} of a vertex $v$, denoted by $d(v)$, is the number of
edges incident to $v$. The {\em out-degree} (respectively, {\em in-degree}) of vertex $v$,
denoted by $d^+(v)$ (respectively, $d^-(v)$), is the number of arcs emanating from (respectively, to) $v$.  If $d^+(v)=d^-(v)=z$ and $d(v)=r$, for all $v \in V$, then $G$ is said to be {\em totally regular\/} of degree $(r,z)$, with $r+z=d$ (or simply {\em
$(r,z)$-regular}).
For mixed graphs, the degree/diameter (optimization) problem is:

\begin{problem}
Given three natural numbers $r,z$
and $k$, find the largest possible number of vertices $N(r,z,k)$ in a
mixed graph with maximum undirected degree $r$, maximum directed out-degree $z$
and diameter $k$.
\end{problem}
This  can be viewed as a generalization of the corresponding problem for undirected
and directed graphs. For these cases, the problem has been widely
studied, see again Miller and \v{S}ir\'a\v{n} \cite{ms13}.
In our general case,  Buset, El Amiri, Erskine, Miller, and P\'erez-Ros\'es \cite{baemp15} proved that the  maximum number of vertices for a mixed graph of diameter $k$ with maximum undirected degree $r$ and maximum out-degree $z$ is:

\begin{equation}
\label{eq:moorebound4}
M(r,z,k)=A\frac{u_1^{k+1}-1}{u_1-1}+B\frac{u_2^{k+1}-1}{u_2-1},
\end{equation}
where,  with $d=r+z$ and $v=(d-1)^2+4z$,
\begin{align}
u_1 &=\displaystyle{\frac{d-1-\sqrt{v}}{2}}, \qquad
u_2 =\displaystyle{\frac{d-1+\sqrt{v}}{2}}, \label{u's}\\
A   &=\displaystyle{\frac{\sqrt{v}-(d+1)}{2\sqrt{v}}}, \quad \
B   =\displaystyle{\frac{\sqrt{v}+(d+1)}{2\sqrt{v}}}. \label{A&B}
\end{align}

This bound applies when $G$ is totally regular with degrees $(r,z)$. Moreover, if we bound the total degree $d=r+z$, the largest number is always obtained when $r=0$ and $z=d$. That is, when the graph is a digraph with no digons (or `edges').
In general, Nguyen, Miller, and Gimbert \cite{nmg07} showed that mixed graphs with $r,z\ge 0$ (or {\em mixed Moore graphs}) only exist for diameter $k=2$.

In some of our constructions we use the Moore bipartite graphs ($r=d=\Delta$ and $z=0$), which are known to exists only for diameters $k=D\in\{2,3,4,6\}$, with number of vertices
\begin{align}
M_b(\Delta,D) &=1+\Delta+\Delta(\Delta-1)+\cdots +\Delta(\Delta-1)^{D-2}+(\Delta-1)^{D-1}
\nonumber\\
 &=2\frac{(\Delta-1)^D-1}{\Delta-2}. \label{Moore-bip}
\end{align}
(Notice that second expression assumes that $\Delta>2$, the case $\Delta=2$ corresponds to a cycle with $M_b(\Delta,D)=2D$ vertices.)


\section{Moore bound for bipartite mixed graphs}

An alternative approach
for computing the bound \eqref{eq:moorebound4} has been given recently by Dalf\'o, Fiol, and L\'opez \cite{dfl16}. In order to study the bipartite case, we now use this last approach.

\begin{proposition}
\label{propo-Moore}
The Moore-like upper bound for bipartite mixed graphs with maximum undirected degree $r$, maximum directed out-degree $z$, and diameter $k$, is
\begin{itemize}
\item[$(a)$]
For $r>0$,
\begin{equation}
\label{Moore-mix-bip}
M_B(r,z,k)=
2\left(A\,\frac{u_1^{k+1}-u_1}{u_1^2-1}+ B\,\frac{u_2^{k+1}-u_2}{u_2^2-1}\right),
\end{equation}
where $u_1$, $u_2$, $A$, and $B$ are given by \eqref{u's} and \eqref{A&B}.
\item[$(b)$]
For $r=0$ $($and $z=d>1$$)$,
\begin{equation}
\label{Moore-bip-dig}
M_b(d,k)=\left\{
\begin{array}{ll}
\displaystyle 2\frac{d^{k+1}-1}{d^2-1}, & \mbox{ for $k$ odd,}\\[.3cm]
\displaystyle 2\frac{d^{k+1}-d}{d^2-1}, & \mbox{ for $k$ even.}
\end{array}
\right.
\end{equation}
\end{itemize}
\end{proposition}

\begin{proof}
Let $G$ be a $(r,z)$-regular bipartite mixed graph with $d=r+z$. Given a vertex $v$, let $N_i=R_i+Z_i$ be the maximum number of vertices at distance
$i(=0,1,\ldots)$ from $v$. Here, $R_i$ is the number of vertices that, in
the corresponding tree rooted at $v$,  have an edge with their parents, and  $Z_i$
is the number of vertices that has an arc from their parents. Then,
\begin{equation}
\label{Ni}
N_i = R_i+Z_i = R_{i-1}((r-1)+z)+Z_{i-1}(r+z).
\end{equation}
That is,
\begin{align}
R_i & = R_{i-1}(r-1)+Z_{i-1}r,  \label{Ri}\\
Z_i & = R_{i-1}z+Z_{i-1}z. \label{Zi}
\end{align}
In matrix form,
$$
\left(
\begin{array}{c}
  R_i \\
  Z_i
\end{array}
\right)=
\left(
\begin{array}{cc}
  r-1 & r\\
  z   & z
\end{array}
\right)
\left(
\begin{array}{c}
  R_{i-1} \\
  Z_{i-1}
\end{array}
\right)=\cdots=\M^i\left(
\begin{array}{c}
  R_{0} \\
  Z_{0}
\end{array}
\right)=
\M^i\left(
\begin{array}{c}
  0 \\
  1
\end{array}
\right),
$$
where $\M=\left(
\begin{array}{cc}
  r-1 & r\\
  z   & z
\end{array}
\right)$, and by convenience $R_0=0$ and $Z_0=1$. Therefore,
\begin{equation}
\label{Ni1}
N_i = R_i+Z_i =\left(
\begin{array}{cc}
  1 & 1
\end{array}
\right)\M^i\left(
\begin{array}{c}
  0 \\
  1
\end{array}
\right).
\end{equation}
Alternatively, note that $N_i$ satisfies an easy linear recurrence formula (see again
Buset, El Amiri, Erskine, Miller, and P\'{e}rez-Ros\'{e}s
\cite{baemp15}). Indeed, from \eqref{Ni} and \eqref{Zi}, we have that
$Z_i=z(N_{i-1}-Z_{i-1})+zZ_{i-1}=zN_{i-1}$. Hence,
\begin{align}
N_i & =(r+z)N_{i-1}-R_{i-1}=(r+z)N_{i-1}-(N_{i-1}-Z_{i-1})\nonumber \\
    & =(d-1)N_{i-1}+zN_{i-2},\qquad i=2,3,\ldots \label{recur}
\end{align}
with initial values $N_0=1$ and $N_1=d$.
Solving the recurrence we get
\begin{equation}
\label{Ni2}
N_i=A u_1^i+B u_2^i\qquad i=0,1,2,\ldots
\end{equation}
where $A$, $B$, $u_1$, and $u_2$ are given by  \eqref{u's} and \eqref{A&B}.
Now, note that, since $G$ is bipartite, it has diameter $k$ if and only if $k$ is
the smallest number such that, for any given vertex $u$, all the vertices of one of
the partite sets of $G$ are at distance at most $k-1$ from $u$. As a consequence, using \eqref{Ni1},
we get the following two cases:
\begin{itemize}
\item[$(i)$] If $k\ge 2$ is even, say $k=2\ell$, then
the Moore bound for a bipartite mixed graph is
\begin{align}
\label{Moore-even}
M_B(r,z,k) =2\sum_{i=1}^{\ell}N_{2i-1}=2\left(
\begin{array}{cc}
  1 & 1
\end{array}
\right)
\sum_{i=0}^{\ell}\M^{2i-1}
\left(
\begin{array}{c}
  0 \\
  1
\end{array}
\right).
\end{align}
\item[$(ii)$] If $k\ge 3$ is odd, say $k=2\ell+1$, then
the Moore bound for a bipartite mixed graph is
\begin{align}
\label{Moore-odd}
M_B(r,z,k) =2\sum_{i=0}^{\ell}N_{2i}=2\left(
\begin{array}{cc}
  1 & 1
\end{array}
\right)
\sum_{i=0}^{\ell}\M^{2i}
\left(
\begin{array}{c}
  0 \\
  1
\end{array}
\right).
\end{align}
\end{itemize}

\begin{figure}[t]
\begin{center}
\includegraphics[width=14cm]{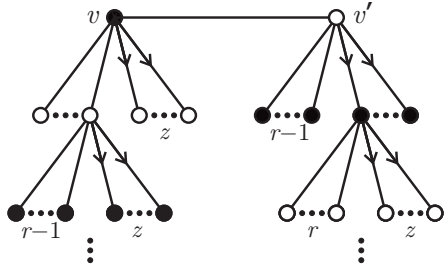}
\vskip-14.5cm
\caption{A Moore $(r,z)$-regular mixed graph hanged from an edge.}
\label{fig1}
\end{center}
\end{figure}

Now, we distinguish two more cases:
\begin{enumerate}
\item[$(a)$]
If $r\ge 1$, that is $G$ is a `proper' mixed graph, we can use \eqref{Ni2} to show that  both cases $(i)$ and $(ii)$, lead to the same expression in \eqref{Moore-mix-bip}.
In fact, in this case, we reach to the same conclusion if we `hang' the graph from an edge $\{v,v'\}$
and for $i=0,1,\ldots,k-1$, let $N_i=R_i+Z_i$ be the maximum number of vertices at distance $i$ from $v$ (respectively, $v'$), such that the respective shortest paths do not contain $v'$ (respectively, $v$), as shown in Figure~\ref{fig1}. Then, solving the recurrence \eqref{recur}, but now with initial values $N_0=1$ and $N_1=d-1$, we get
$N_i=A'u_1^i+B'u_2^i$, for $i=0,1,2,\ldots$, with $u_1$ and $u_2$ given by \eqref{u's},
$A'=\frac{\sqrt{v}-(d-1)}{2\sqrt{v}}$, and
$B'=\frac{\sqrt{v}+(d-1)}{2\sqrt{v}}$. Thus,
$M_B(r,z,k) =2\sum_{i=1}^{k-1}N_{i}$ yields again \eqref{Moore-mix-bip}.
\item[$(b)$]
Otherwise, if $r=0$, $G$ is a digraph (with no digons), $z=d$, and \eqref{u's}, \eqref{A&B}
give $A=0$, $B=1$, and $u_2=d$. Hence, \eqref{Ni2} turns out to be $N_i=d^i$, and the formulas in $(i)$ and $(ii)$ become the Moore bounds \eqref{Moore-bip-dig} for bipartite digraphs  given by Fiol and Yebra in \cite{fy90} for $d>1$.
(The case $d=1$ corresponds to a directed cycle with order $M_b(1,k)=k+1$.)
\end{enumerate}
This completes the proof.
\end{proof}

The remaining particular case when $G$ is an (undirected) graph, that is, $z=0$ and $r=d$, is already included in  the statement of the proposition. Indeed, in such a case, \eqref{u's} and \eqref{A&B} give $u_1=0$, $u_2=d-1$, $B=\frac{d}{d-1}$, and \eqref{Moore-mix-bip}
with $d=\Delta$ and $k=D$ becomes \eqref{Moore-bip}.

It is also worth noting that the eigenvalues of  the matrix $\M$  in \eqref{Ni} are precisely $u_1$ and $u_2$ given in \eqref{u's}.
Then, we have $\M^{i}=\V\D^{i}\V^{-1}$, where $\D=\diag(u_1,u_2)$, and
$$
\V=\left(\begin{array}{cc}
\frac{1-r+z+\sqrt{v}}{2z} & \frac{1-r+z-\sqrt{v}}{2z}\\ \\
1 & 1
\end{array}
\right).
$$

The bounds of Proposition \ref{propo-Moore} apply when $G$ is totally regular with degrees $(r,z)$. For instance, the Moore bounds for diameters $k=2,3,4$ and total degree $d=r+z$ turn out to be
\begin{align}
M_B(r,z,2) &=2d,\label{k=2}\\
M_B(r,z,3) &=2(d^2-r+1), \label{k=3}\\
M_B(r,z,4) &=2(d^3-d^2+(z-r+1)d+r). \label{k=4}
\end{align}
Moreover, if we bound the total degree $d=r+z$, the largest number always is
obtained when $r=0$ and $z=d$. That is, when the mixed graph is, in fact, a digraph with no edges. In the following table we show the values of \eqref{Moore-even} and
\eqref{Moore-odd} when $r=d-z$, with $0\le z\le d$, for different values of $d$ and diameter $k$. In particular, when $z=0$, the bound corresponds to the Moore
bound for bipartite graphs (numbers in boldface).

\medskip
{\footnotesize{
\medskip\noindent
\begin{tabular}{@{\,}ccccccc@{\,}}
$k\diagdown d$ \!\!\!\!\!\!  & $\vline$ & $1$ & $2$ & $3$ & $4$ & $5$ \\
\hline
2& \vline  &   {\bf 2}   &     {\bf 4}    &   {\bf  6} &       {\bf 8}   &  {\bf
10} \\
3& \vline & $2z+{\bf 2}$ & $2z+{\bf 6}$ &  $2z+{\bf 14}$ & $2z+{\bf 26}$ & $2z+{\bf
42}$ \\
4& \vline & $2z+{\bf 2}$ & $6z+{\bf 8}$ & $10z+{\bf 30}$ &$14z+{\bf 80}$ &
$18z+{\bf 170}$ \\
5& \vline & $2z^2+2z+{\bf 2}$ & $2z^2+12z+{\bf 10}$ & $2\z^2+34z+{\bf 62}$
&$2z^2+68z+{\bf 242}$ & $2z^2+114z+{\bf 682}$ \\
6& \vline & $2z^2+2z+{\bf 2}$ & $8z^2+20z+{\bf 12}$ & $14z^2+98z+{\bf 126}$
&$20z^2+284z+{\bf 728}$ & $26z^2+626z+{\bf 2730}$\\
\hline
\end{tabular}
\medskip}}

\section{Mixed bipartite Moore graphs}

Mixed bipartite graphs attaining bound \eqref{Moore-mix-bip} will be referred as {\em mixed bipartite Moore graphs}. These extremal graphs have been widely studied for the undirected case, where they may only exist for diameters $k \in \{2,3,4,6\}$ (Feit and Higman \cite{feit64}). For $d=2$, even cycles $C_{2k}$ are mixed Moore graphs. For any $d \geq 3$, complete bipartite graphs $K_{d,d}$ are the unique bipartite Moore graphs of diameter $k=2$ and degree $d$. Nevertheless, the problem is not closed for $k \in \{3,4,6\}$, where bipartite Moore graphs have been constructed only when $d-1$ is a prime power $p^l$ ($k=3,4$) and $3^{2l+1}$ ($k=6$). For the directed case, the problem of the existence of bipartite Moore digraphs was settled by Fiol and Yebra \cite{fy90}. These digraphs may only exist for $k \in \{2,3,4\}$. As for the undirected case, complete bipartite digraphs are the unique bipartite Moore digraphs of diameter $k=2$ meanwhile for $k \in \{3,4\}$ some families of bipartite Moore digraphs have been constructed, although the problem of their enumeration is not closed (see again \cite{fy90}, and Fiol, Gimbert, G\'omez and Wu \cite{fggw03}).

\begin{figure}[t]
    \begin{center}
        \includegraphics[width=12cm]{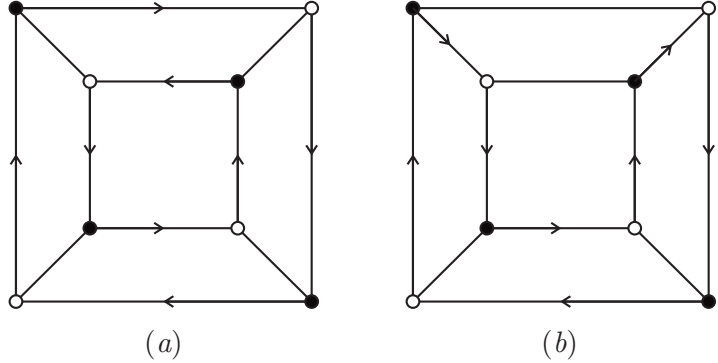}
    \vskip-10.75cm
	\caption{The only two mixed bipartite Moore graphs of parameters $r=z=1$ and $k=3$.}
	\label{fig2}
\end{center}
\end{figure}

Heading into the mixed case, from here on, we suppose that our mixed graphs contains at least one edge ($r \geq 1$), one arc ($z \geq 1$) and the diameter $k$ is at least three. It is easy to see that there are exactly two non-isomorphic mixed bipartite Moore graphs for $r=z=1$ and $k=3$ (see Figure~\ref{fig2}). Note that the mixed graph depicted in Figure~\ref{fig2}$(a)$ is the line digraph of the cycle $C_4$ $($seen as a digraph, so that each edge corresponds to a digon\/$)$. It is also the Cayley graph of the dihedral group $D_4=\langle a,b\, |\, a^4\!=\!b^2\!=\!(ab)^2\!=\!1 \rangle$, with generators $a$ and $b$. The spectrum of this mixed graph is the same as the cycle $C_4$  plus four more $0$'s, that is, $\spec G=\left\{2,\  0^6,\ -2\right\}$. (This is because $G$ is the line digraph of $C_4$, see Balbuena, Ferrero, Marcote, and Pelayo \cite{bfmp03}.) In fact, the mixed graph of Figure~\ref{fig2}$(b)$ is cospectral with $G$, and it can be obtained by applying a recent method to obtain cospectral digraphs (see Dalf\'o and Fiol \cite{df16}). Both mixed graphs are partial orientations of the hipercube graph of dimension $3$ that preserve the diameter of the underlying graph.
In contrast with that, we next prove that mixed bipartite Moore graphs do not exist for larger diameters.

\begin{figure}[t]
\begin{center}
\includegraphics[width=16cm]{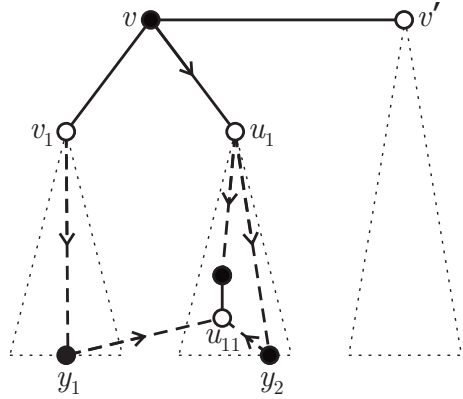}
\vskip-15cm
\caption{The structure of a mixed bipartite Moore graph of odd diameter $k \geq 5$. Here it is depicted for $r=2$ and $z=1$.}
\label{fig3}
\end{center}
\end{figure}

\begin{proposition}
Mixed bipartite Moore graphs do not exist for any $r\geq 1$, $z \geq 1$ and $k \geq 4$.
\end{proposition}

\begin{proof}
First of all, we recall that there is a unique shortest path between any ordered pair of vertices of distance less than $k$ in a mixed bipartite Moore graph. Let $G$ be a mixed bipartite Moore graph and let us `hang' $G$ from any edge $\{v,v'\}$. We define the set $\Gamma_{j}(v)$, for $j<k$, as the set of vertices at distance $j$ from $v$ such that the unique mixed path from $v$ does not pass through $v'$, that is, $\Gamma_{j}(v)$ represents those vertices in the tree `hanged down' by $v$ at distance $j$ from $v$. Let $\{v_1,\dots,v_{r-1}\}$ be the set (possibly empty if $r=1$) of vertices adjacent from $v$ by an edge (others than $v'$). Let $\{u_1,\dots,u_z\}$ be the set of vertices adjacent from $v$ by an arc. Let $u_{11}$ be any vertex at distance $k-3$ from $u_1$ such that $u_{11}$ is connected by an edge through its antecessor in the unique $u_1$-$u_{11}$ mixed path joining them (see Figure~\ref{fig3}). Since $G$ is $(r,z)$-regular, there must exist $z$ vertices $\{y_1,\dots,y_z\}$ such that $y_i \rightarrow u_{11}$ is an arc for all $i=1,\dots,z$. Notice that either $y_i \in \Gamma_{k-1}(v)$ or $y_i \in \Gamma_{k-1}(v')$. Nevertheless, $u_{11}$ belongs to the same vertex partition as the vertices in $\Gamma_{k-1}(v')$ and as a consequence $y_i \in \Gamma_{k-1}(v)$ for all $i$. Notice that $\bigcup_{j=1}^{r-1} \Gamma_{k-2}(v_j)$ and $\bigcup_{j=1}^{z} \Gamma_{k-2}(u_j)$ are a partition of the vertex set $\Gamma_{k-1}(v)$. Now, if one $y_i$ belongs to $\bigcup_{j=1}^{r-1} \Gamma_{k-2}(v_j)$ (say $y_1 \in \Gamma_{k-2}(v_1)$), then we have two different paths from $v_1$ to $u_{11}$ of length $<k$: One path passing through $v$ and the other one passing through $y_1$, which is impossible. So, we can assume that $y_i \in \bigcup_{j=1}^{z} \Gamma_{k-2}(u_j)$ for all $i=1,\dots,z$. Now, since all the vertex sets $\Gamma_{k-2}(u_j)$ are disjoint, there is exactly one $y_i$ in each set $\Gamma_{k-2}(u_j)$ or there is a set $\Gamma_{k-2}(u_j)$ containing at least two $y_i$'s. In the first situation, we have that there is one $y_i$ (say $y_2$) into $\Gamma_{k-2}(u_1)$ but then we have two different paths from $u_1$ to $u_{11}$ of length $<k$ (one shortest path of length $k-3$ and the other of length $k-1$ through $y_2$). In the second situation, we have at least two $y's$ belonging to the same set (say $\Gamma_{k-2}(u_j)$) and then we have again two different paths from $u_j$ to $u_{11}$ of length $k-1$, which is a contradiction.
\end{proof}

Because of the previous result, it would be interesting to study
{\em mixed bipartite almost Moore graphs}, that is, mixed bipartite graphs with order $M(r,z,k)-2$.
Note that, in contrast with general mixed almost Moore graphs (see L\'opez and Miret \cite{lm16}), in the bipartite case
the mixed almost Moore graphs have two vertices less than the Moore bound. In Figure \ref{fig4}, we show two examples of
mixed bipartite almost Moore $(1,1)$-regular graphs with diameter $k=4$ and order $M(1,1,4)-2=12$.

\begin{figure}[t]
\begin{center}
\includegraphics[width=12cm]{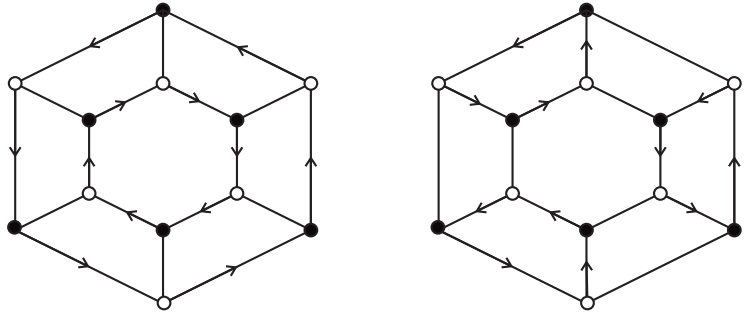}
\vskip-11.25cm
\caption{Two cospectral mixed bipartite almost Moore $(1,1)$-regular graphs.}
\label{fig4}
\end{center}
\end{figure}

\subsection{The case of diameter three}
Since mixed bipartite Moore graphs do not exist for diameter $k\ge 4$, we now investigate the case of diameter $k=3$. We first show that, in contrast with the case of mixed Moore graphs of diameter $k=2$, where the possible spectrum imposes conditions on $r$ and $z$ (see Bos\'ak \cite{b79}), a mixed bipartite Moore graph has always a simple spectrum.

\begin{proposition}
Let $G$ be a mixed bipartite Moore graph with degrees $r$, $z$, with $d=r+z$, and number of vertices $N=2(d^2-r+1)$. Then, $G$ has spectrum
\begin{equation}
\label{spec-k=3}
\spec G=\{d^1, (\sqrt{r-1})^{d^2-r}, (-\sqrt{r-1})^{d^2-r}, -d^1\},
\end{equation}
where the superscripts stand for the multiplicities.
\end{proposition}
\begin{proof}
Let $\A$ be the adjacency matrix of $G$. Given any two vertices $u,v$, there exists a unique shortest path from $u$ to $v$ if $\dist(u,v)\le 2$, and exactly $d$ shortest paths if $\dist(u,v)=3$. Then, $\A$ satisfies the matrix equation
$$
\I+(\A^2-r\I)+\frac{1}{d}\left(\A^2-(r-1)\I\right)\A=\J,
$$
where $\J$ is the all-$1$ matrix. Then, since such a matrix has eigenvalues $N$, with multiplicity 1, and $0$ with multiplicity $N-1$, the eigenvalues of $\A$ are the solutions of the equations $H(x)=N$ and $H(x)=0$, where $H$ is the Hoffman-like polynomial
\begin{equation}
\label{H}
H=\frac{1}{d}(x^3-(r-1)x)+x^2+1-r.
\end{equation}
The first equation has solution $x=d$, as expected since $N=H(d)=2(d^2-r+1)$, whereas the second one holds for $x=\pm \sqrt{r-1}$ and $x=-d$. Finally, the multiplicities come from the fact that $G$ is bipartite and they must add up to $N$.
\end{proof}

Notice the particular case $r=1$, where according to \eqref{spec-k=3}, the mixed graphs have only three distinct eigenvalues: $d$ and $-d$ with multiplicity $1$, and $0$ with multiplicity $2(d^2-1)$. This is precisely the case of the following family:

\begin{proposition}
Moore bipartite mixed graphs with diameter $k=3$ and $r=1$ exist for any value of $z\ge 1$.
\end{proposition}

\begin{proof}
For any given integer $d\ge 2$, let us consider the line digraph of the complete bipartite graph (seen as a symmetric digraph) $G=L(K_{d,d})$ (for example, Figure~\ref{fig2}$(a)$ shows the case $d=2$). Then, as $G$ has $d^2$ digons and diameter $2$, its line digraph $G$ is a bipartite mixed graph with diameter $3$ (according to Fiol, Yebra, and Alegre \cite{fya84}), $r=1$ (digons correspond to edges),
$z=d-1$, and number of vertices $2d^2$, so attaining the Moore bound \eqref{k=3}.
\end{proof}

Moreover, the minimum polynomial of $G$ is $m(x)=x^4-d^2x^2$.
Then, $G$ is a weakly distance-regular digraph (according to Comellas, Fiol, Gimbert and Mitjana \cite{cfgm04}) with distance polynomials
\begin{equation*}
p_0=1,\quad
p_1 =x,\quad
p_2 =x^2-1,\quad
p_3 =\frac{1}{d}x^3-x,
\end{equation*}
and Hoffman polynomial $H=\sum_{i=0}^3 p_i=\frac{1}{d}(x^3-x)+x^2+x$, as shown in  \eqref{H}.

For some other values of the diameter, we also have some families of mixed graphs that are asymptotically dense:
\begin{proposition}
There exist families of bipartite mixed graphs with diameter $k=4,5,7$, and $r=1$, that asymptotically attain the Moore bound for large values of $z$ being a power of a prime minus one.
\end{proposition}

\begin{proof}
Just consider, as in the previous proof, the line digraph of the corresponding Moore bipartite graphs that exists for diameters $D=3,4,6$. For example, a Moore bipartite graph with diameter $D=3$ that exists for degree $\Delta=p^l +1$ has order $2 \frac{(\Delta-1)^3-1}{\Delta-2}$ (see \eqref{Moore-bip}). Its corresponding line digraph is a mixed bipartite graph with parameters $k=4$, $r=1$ and $z=\Delta-1$ with order $2d \frac{(d-1)^3-1}{d-2}$. This mixed bipartite graph, for $d=r+z$ large enough, attains the corresponding Moore bound given by \eqref{k=4}, that is,  $M_B(1,d-1,4)=2(d^3-d^2+(d-1)d+1)$. The cases for diameters $k=5,7$ are similar.
\end{proof}

\noindent{\large \bf Acknowledgments.}  This research is supported by the
{\em Ministerio de Ciencia e Innovaci\'on}, and the {\em European Regional Development Fund} under project MTM2014-60127-P, the {\em Catalan Research Council} under project 2014SGR1147 (C. D. and M. A. F.). The author N. L. has been supported in part by grant MTM2013-46949-P, from {\em Ministerio de Econom\'{\i}a y Competitividad}, Spain.

\end{document}